\documentclass[12pt]{amsart}

\usepackage{amsrefs}

\usepackage{graphicx}
\usepackage{amsfonts}
\usepackage{amssymb}
\usepackage{amsmath}
\usepackage{amscd}
\usepackage{amsthm}
\usepackage{tikz}
\usepackage{tikz-cd}
\usepackage{hyperref}
\usepackage[all]{xy}
\usetikzlibrary{matrix}
\usepackage{pinlabel}
\usepackage{color}
\usepackage[margin=3cm, marginpar=2.5cm]{geometry}

\newtheorem{theorem}{Theorem}[section]
\newtheorem{proposition}[theorem]{Proposition}
\newtheorem{lemma}[theorem]{Lemma}
\newtheorem{corollary}[theorem]{Corollary}

\theoremstyle{remark}
\newtheorem{remark}[theorem]{Remark}
\newtheorem{definition}[theorem]{Definition}

\newcommand{\s}{\mathfrak{s}}
\renewcommand{\t}{\mathfrak{t}}

\newcommand{\Z}{\mathbb{Z}}

\newcommand{\Q}{\mathbb{Q}}
\newcommand{\R}{\mathbb{R}}
\newcommand{\RP}{\mathbb{RP}}

\newcommand{\Spinc}{\text{Spin}^c}

\def\co{\colon\thinspace}
\newcommand{\HFp}{\HF^+}
\newcommand{\HFoo}{\HF^\infty}

\DeclareMathOperator{\rk}{rk}

\DeclareMathOperator{\kerr}{ker}
\DeclareMathOperator{\cokerr}{coker}
\DeclareMathOperator{\HF}{HF}

\DeclareMathOperator{\Char}{Char}

\title{$3$-manifolds that bound no definite $4$-manifold}
\author{Marco Golla}
\address{CNRS, Laboratoire de Math\'ematiques Jean Leray, Nantes, France}
\email{marco.golla@univ-nantes.fr}

\author{Kyle Larson}
\address{Department of Mathematics, University of Georgia, Athens, Georgia U.S.A.}
\email{kyle.larson@uga.edu}

\begin{document}





\begin{abstract}
We produce a rational homology 3-sphere that does not smoothly bound either a positive \emph{or} negative definite 4-manifold.
Such a 3-manifold necessarily cannot be rational homology cobordant to a Seifert fibered space or any 3-manifold obtained by Dehn surgery on a knot. The proof requires an analysis of short characteristic covectors in bimodular lattices.
\end{abstract}

\maketitle


\vspace{-1 \baselineskip}

\begin{section}{Introduction}\label{introduction}
By a definite 4-manifold we mean  a 4-manifold whose intersection form is positive or negative definite.
Many familiar classes of 3-manifolds are known to bound either positive or negative definite 4-manifolds. 
These include, for example, any Seifert fibered rational homology sphere or any 3-manifold obtained by nonzero surgery on a knot in $S^3$.
Lens spaces and those 3-manifolds which are the double covers of $S^3$ branched over an alternating knot are examples of 3-manifolds that bound positive \emph{and} negative definite 4-manifolds.

However, one can often argue that a given 3-manifold cannot smoothly bound \emph{both} positive and negative definite 4-manifolds (all manifolds in this paper are assumed to be smooth and oriented).
For example, consider the Poincar\'e homology sphere $P$, oriented as the link of the Brieskorn singularity.
Since its Fr{\o}yshov $h$-invariant (or equivalently, the Heegaard Floer correction term $d$) is positive, it cannot bound a positive definite 4-manifold (see \cite[Theorem 3]{Froyshov1} or \cite[Corollary 9.8]{OSz}).
On the other hand, $P$ does bound negative definite 4-manifolds, e.g., either as $-1$-surgery on the left-handed trefoil or as the boundary of the negative $E_8$-plumbing.

To obstruct a 3-manifold from bounding \emph{either} positive or negative definite 4-manifolds is rather more difficult.
Fr{\o}yshov has announced the first examples of integral homology spheres that bound no definite 4-manifolds, provided that the first homology of the putative definite 4-manifold does not contain 2-torsion \cite{Froyshov2}.
His argument depends on the fact that the $h$-invariant and the $q_2$-invariant arising from instanton homology are linearly independent.
%
Recent work of Nozaki, Sato, and Taniguchi~\cite{NST}, using filtered instanton Floer homology, gives examples of integral homology spheres that bound no definite 4-manifolds, without any restrictions on the torsion in homology.
Any \emph{rational} homology sphere rationally cobordant to one of the examples of~\cite{NST} evidently has the same property.

The main result of this paper is the following.

\begin{theorem}\label{main}
There exist rational homology spheres that are not cobordant to an integral homology sphere and that bound no definite $4$-manifold.
\end{theorem}

The authors understand that Mathew Hedden has produced similar (unpublished) examples.
As far as we understand, such a result is not currently accessible with the techniques of~\cite{NST}, since filtered instanton Floer homology is only defined for integral homology spheres.

The proof of our theorem instead combines an inequality of Ozsv\'ath and Szab\'o~\cite{OSz} relating correction terms and the squares of first Chern classes of spin$^c$ structures together with an analysis of short characteristic covectors in bimodular lattices, using work of Elkies~\cite{Elkies} on unimodular lattices.

Let us give a specific example: let $N$ denote the 3-manifold $N := 3P \# \overline{Y}$, i.e.~the sum of three copies of the Poincar\'e sphere $P$ and the Seifert fibered space $\overline{Y} = Y(2;\frac{15}{13},\frac{17}{3},\frac{23}{22})$.
(See Figure~\ref{f:seifert}.)
Note that $H_1(N) \cong \Z/2\Z$, so $N$ cannot be homology cobordant to an integral homology sphere. We will prove below that $N$ cannot bound a definite 4-manifold.

One can obtain other examples satisfying various properties by constructing 3-manifolds that are integral or rational homology cobordant to $N$.
For example, work of Livingston~\cite{Livingston} implies that $N$ is integral homology cobordant to an irreducible 3-manifold, which can further be assumed to be hyperbolic by work of Myers~\cite{Myers}.
From what we discussed above we get the following immediate corollaries.

\begin{figure}
\labellist
\pinlabel $2$ at 138 46
\pinlabel $\frac{15}{13}$ at 13 7
\pinlabel $\frac{17}{3}$ at 49 7
\pinlabel $\frac{23}{22}$ at 86 7
\endlabellist
\centering
\includegraphics[scale=1.5]{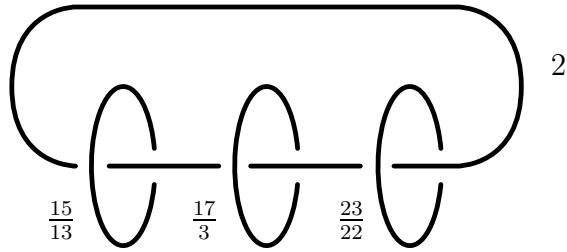}
\caption{A surgery description of $Y(2;\frac{15}{13},\frac{17}{3},\frac{23}{22})$.}
\label{f:seifert}
\end{figure}

\begin{corollary}
$N$ is not rational homology cobordant to any Seifert fibered space.
\end{corollary}

Examples of integral homology spheres that are not integral homology cobordant to any Seifert fibered space have been produced by Stoffregen~\cite{Stoffregen} and Fr{\o}yshov~\cite{Froyshov2}; previously, rational homology spheres that are not \emph{integral} homology cobordant to any Seifert fibered space appeared in~\cite{CochranTanner}.

\begin{corollary}
$N$ is not rational homology cobordant to any $3$-manifold obtained by Dehn surgery on a knot in $S^3$.
\end{corollary}

It is often quite difficult to prove that a 3-manifold cannot be obtained by surgery on a knot (see, for example,~\cite{HKL}), and most of the known obstructions are not preserved under rational homology cobordism.
However, in~\cite{HKMP} examples of 3-manifolds $Y$ with $b_1(Y)=1$ are constructed that are not rational homology cobordant to 0-surgery on a knot.

In Section~\ref{example} we will give a more direct proof that the 3-manifold $\overline Y$ is not obtained as Dehn surgery along a knot in $S^3$, and use $\overline Y$ to produce another example of a \emph{spineless} 4-manifold (see~\cite{LevineLidman} and~\cite{HaydenPiccirillo} for previous work on the subject---Hayden and Piccirillo's results in particular are much stronger than ours).

\subsection*{Organization of the paper}
In Section~\ref{lattices} we study bimodular lattices and their short characteristic covectors.
In Section~\ref{obstruction} we describe an obstruction for a rational homology sphere $Y$ with $H_1(Y)\cong \Z/2\Z$ to bound a definite 4-manifold.
In Section~\ref{example}, we show that the manifold $N$ described above satisfies the conditions of the obstruction from Section \ref{obstruction}.

\subsection*{Acknowledgements}
We thank Paolo Aceto for bringing this problem to our attention, and Adam Levine, Brendan Owens, Kim Fr{\o}yshov, Matthew Hedden, and Andr\'as Stipsicz for useful conversations.
MG thanks the R\'enyi Institute for their hospitality at the beginning of this project.

\end{section}

\begin{section}{Characteristic covectors of bimodular lattices}\label{lattices}

In this section, a lattice $\Lambda$ will be a subset $\Lambda \subset \R^m$ isomorphic to $\Z^m$, and such that, with respect to the Euclidean scalar product on $\R^m$ one has $v\cdot w \in \Z$ for each $v,w \in \Lambda$. A lattice is said to be \emph{minimal} if it contains no vector of square $1$.

We denote with $\Lambda'$ the \emph{dual lattice} of $\Lambda$, i.e. the set of vectors in $\R^m$ that pair integrally with each element in $\Lambda$. Note that $\Lambda \subset \Lambda'$, and that $\Lambda'$ is \emph{not} a lattice according the definition above, unless $\Lambda = \Lambda'$. (There are always vectors in $\Lambda'$ that square to a rational if the containment is strict.) If $\Lambda = \Lambda'$ we say that $\Lambda$ is \emph{unimodular} (or \emph{self-dual}). Even if $\Lambda$ and $\Lambda'$ both live in $\R^m$ and we use the scalar product on $\R^n$ for the pairing, we use the notation $\langle \xi, v\rangle$ to denote the pairing of a covector $\xi \in \Lambda'$ and a vector $v \in \Lambda$.

We denote with $\Char(\Lambda)$ the subset of $\Lambda'$ comprising all $\xi$ such that $\langle \xi, v\rangle \equiv v^2 \pmod 2$. Each such $\xi$ is a \emph{characteristic covector}. (Again, when $\Lambda$ is unimodular $\xi$ is actually an element in $\Lambda$, but we prefer to talk about covectors to emphasize that we are thinking about the dual.)

Throughout the section, the letter $L$ will always denote a positive definite integral lattice of rank $n$ and determinant $2$ (a \emph{bimodular lattice}), $A$ will be an auxiliary lattice of determinant $2$, and $M_A$ will be the lattice $L \oplus A$.
Note that $M'_A/M_A \cong L'/L \oplus A'/A \cong \Z/2\Z \oplus \Z/2\Z$ has a unique metabolizer (i.e. a subgroup isomorphic to $\Z/2\Z$ that is isotropic with respect to the induced $\Q/\Z$-valued bilinear form induced by the product on $M_A$), so that $M_A$ is an index-$2$ sublattice of a unimodular lattice $U_A$.
(See, for example,~\cite{Larson} for a topologically-minded treatment.)
In fact, $U_A$ is simply $L\oplus A \cup (L'\setminus L) \times (A'\setminus A) \subset L' \oplus A'$.
With a slight abuse of notation, we view $L$ and $A$ as subsets of $U_A$. It is easy to see that $L = A^\perp$ and $A = L^\perp$.
Since neither $L$ nor $A$ have a metabolizer, note that $L\otimes \Q \cap U_A = L$ and $A \otimes \Q \cap U_A = A$.
Note that $U_A$ is uniquely determined by both $L$ and $A$, but we do not make the dependency on $L$ explicit in the notation.

Dualizing, $U_A' \cong U_A$ is an index-2 subset of $M_A' = L' \oplus A'$. Moreover, the restriction maps $U_A' \to L'$ and $U_A' \to A'$ are both onto. In the same way, the restriction maps $\Char(U_A) \to \Char(L)$ and $\Char(U_A) \to \Char(A)$ are onto.

Two choices stand out. We can choose $A = L$, or $A = A_1$, where $A_1$ is the rank-$1$ lattice generated by a vector of square $2$ (i.e. a \emph{root}). We write $U$ and $M$ instead of $U_{A_1}$ and $M_{A_1}$.
Call $r$ one of the two generators of the auxiliary lattice $A = A_1$. 
By construction, $U$ contains $r$, a vector of norm $2$.

\begin{lemma}\label{l:charcongruence}
Let $\xi \in \Char(L)$ be a characteristic covector. Then $\xi^2$ is an integer, and $\xi^2 \equiv n\pm 1 \pmod 8$.
\end{lemma}

\begin{proof}
Choose a characteristic covector $\xi_U \in \Char(U)$ that extends $\xi$, and call $\xi_A \in \Char(A_1)$ its restriction to $A_1$. Then $\xi_U^2 \equiv \rk U = n+1 \pmod 8$, and $\xi_U^2 = \xi^2 + \xi_A^2$. Since $\xi_A^2 \equiv 0,2 \pmod 8$ by direct verification, the lemma follows.
\end{proof}

We call $\Char_\pm(L)$ the set of characteristic covectors $\xi$ of $L$ for which $\xi^2 \equiv n\pm 1 \pmod 8$.
Let us go back to the case of $A$ arbitrary. Since $A$ is a determinant-$2$ lattice, $\Char_\pm(A)$ are defined, too.

\begin{lemma}\label{l:extend}
$\xi_u \in U_A'$ is a characteristic covector of $U_A$ if and only if there exists a sign $s$ such that $\xi_{u}|_L \in \Char_s(L)$ and $\xi_u|_{A} \in \Char_{-s}(A)$.
\end{lemma}

Recall that if $\Lambda$ is a lattice, then $\Char(\Lambda)$ is affine over $2\Lambda'$ (i.e. $\Char(\Lambda) = \xi + 2\Lambda'$ for any $\xi \in \Char(\Lambda)$). An easy extension of this fact is the observation that $\Char_{\pm}(L)$ is affine over $2L$ (and not over $2L'$), and that translations by elements in $2L'\setminus 2L$ swap $\Char_+(L)$ and $\Char_-(L)$.

\begin{proof}
Call $\xi_\ell = \xi_u|_L$ and $\xi_a = \xi_u|_{A}$.
The `only if' direction is clear, since $\xi_\ell^2 + \xi_a^2 = \xi_u^2 \equiv n+\rk A \pmod 8$.

Let us look at the `if' direction. Now, let $C := \Char_+(L) \times \Char_-(A) \cup \Char_-(L) \times \Char_+(A) \subset M'$.
By the `only if' direction above, $C$ contains $\Char(U_A)$, which is an affine space over $2U_A'$.
On the other hand, $\Char_\pm(L) \times \Char_\mp(A)$ is also affine over $2L+2A = 2M_A$, and, by the remark below the statement, if $v \in U_A \setminus (L\oplus A) = (L'\setminus L) \times (A'\setminus A)$ and $\xi \in \Char_\pm(L) \times \Char_\mp(A)$ then $\xi + 2v \in \Char_\mp(L) \times \Char_\pm(A)$, so that $C$ is affine over $2U_A'$, too.
In particular, $C$ and $\Char(U_A)$ are both affine subspaces over $2U_A'$, and since $\Char(U_A) \subset C$ then they are equal.
\end{proof}

We denote with $I^m$ the lattice $\Z^m$ with the diagonal intersection form (i.e. with an orthonormal basis). The lattice $\Delta_n := I^{n-1} \oplus A_1$ is the unique bimodular lattice of rank $n$ whose intersection form is diagonal.

\begin{lemma}\label{l:diagonal}
The lattice $U$ is diagonal if and only if $L$ is.
\end{lemma}

\begin{proof}
Suppose that $U$ is diagonal, with base $e_0,\dots,e_n$. Then the generator $r$ of $A$ is a root in $U$. In particular, up to re-indexing the generators of $U$ (and possibly flipping signs), $r = e_0 - e_1$. Now, $L \cong \langle r \rangle^\perp$, but $\langle r \rangle^\perp$ is spanned by $e_0 + e_1, e_2, \dots, e_n$, and in particular it is isomorphic to $\Delta_n$.

If $L = \Delta_n$, then $L$ embeds in $I^{n+1}$ as we have just seen; however, $U$ is uniquely determined by $L$, so $U \equiv I^{n+1}$.
\end{proof}

We recall a result of Elkies on characteristic covectors in unimodular lattices.

\begin{theorem}[\cite{Elkies}]\label{t:elkies1}
Let $\Lambda$ be unimodular lattice of rank $m$. Then
\[
\min_{\xi \in \Char(\Lambda)} \xi^2 \le m,
\]
and the equality is attained if and only if $\Lambda$ is the diagonal lattice $I^m$.
\end{theorem}

We find it convenient to introduce the notation for the defect of a lattice.
The \emph{defect} $d(\Lambda)$ of $\Lambda$ is:
\[
d(\Lambda) = \min_{\xi \in \Char(\Lambda)} \frac{\xi^2- \rk\Lambda}4 
\]
Note that Elkies' theorem can be rephrased as saying that a unimodular lattice has non-positive defect, and that the defect is $0$ if and only if the lattice is diagonal.

When $L$ is bimodular, using Lemma~\ref{l:charcongruence} we can identify \emph{two} defects, denoted with $d_\pm(L)$:
\[
d_\pm(L) = \min_{\xi \in \Char_\pm(L)} \frac{\xi^2-\rk L}4 
\]
Note that $d_\pm(L) \equiv \pm\frac14 \pmod 2$.

The main result of this section is a bimodular version of of Elkies' theorem and gives another characterization of $\Delta_n$ (see~\cite{OwensStrle-char} for a different characterization).

\begin{theorem}\label{p:2-elkies}
For every bimodular lattice $L$, $d_+(L) + d_-(L) \le 0$. Moreover, if equality is attained, $L$ is diagonal, and in particular $d_\pm(L) = \pm\frac14$.
\end{theorem}

As mentioned above, a root in a lattice $\Lambda$ is a vector of square $2$; we say that $\Lambda$ is a \emph{root lattice} if it is rationally spanned by its set of roots. To a collection $R \subset \Lambda$ of $\Q$-linearly independent roots we associate an edge-weighted graph $G(R)$ as follows: the set of vertices of $G(R)$ is $R$, and there is an edge joining $r, s \in R$ with weight $r\cdot s$ if $r \cdot s \neq 0$. Note that if $r, s \in R$ are distinct, they are linearly independent, so by Cauchy--Schwarz $r\cdot s \in \{-1,0,1\}$.

We will make the choice $A = L$ in the proof. To make the notation lighter, we denote $U_L$ by $D$, and we call it the \emph{double} of $L$. To distinguish between the two summands in $M_A \cong L \oplus L$, we still call $A$ its second summand; however, we drop the dependency on $A$ from the notation. In summary, we have $M = L\oplus A$ as an index-$2$ sublattice of $D$, and $L$ and $A$ are viewed as a pair of orthogonal sublattices in $M$ and in $D$.

We call $\pi_\ell \co M \to L$ and $\pi_a \co M \to A$ the two orthogonal projections, and $\rho_\ell \co D' \to L'$ and $\rho_a \co D' \to A'$ the two restriction maps. In fact, $\pi_\ell$ and $\rho_\ell$ are both restrictions of a linear map $D\otimes \Q \to L \otimes \Q$ (and similarly for $A$).

\begin{proof}
We can assume, without loss of generality, that $L$ is minimal (i.e. it contains no vectors of norm $1$): indeed, it is easy to verify that $d_\pm(L) = d_\pm(L \oplus I^m)$, and $L$ is diagonal if and only if $L\oplus I^m$ is.

Call $n$ the rank of $L$. Consider now $D$, the double of $L$. $D$ is a unimodular lattice, so $d(D) \le 0$. By Lemma~\ref{l:extend}, $d(D) = \min\{d_+(L) + d_-(A), d_-(L) + d_+(A)\} = d_+(L) + d_-(L)$, which proves the first assertion.

Let us now suppose that $d_+(L) + d_-(L) = 0$; again by Elkies' theorem, this implies that $D$ is the diagonal lattice $I^{2n}$. Call $e_1, \dots, e_{2n}$ an orthonormal basis of $D$. 

We claim that $L$ is a root lattice.

Since $L$ is minimal, $e_i \in D\setminus M$. However, since $M$ has index $2$ in $D$, $2e_i \in M$. We also know that $2e_i \not\in L \cup A$: indeed, as mentioned at the beginning of the section, $L \otimes \Q \cap D = L$, so if $2e_i \in L$, then also $e_i \in L$. (By symmetry, this proves the statement for $A$ as well.)

This implies that $r_i := \pi_\ell(2e_i)$ and $\pi_a(2e_i)$ are two non-zero vectors whose squares sum to $(2e_i)^2 = 4$; since $L$ is minimal, they both have square $2$. Since the collection $\{2e_i\}$ is a rational basis of $D$, the collection $\{r_i\}$ is a set of roots that rationally spans $L$, which proves the claim.

If $d_\pm(L) = \pm\frac14$, then, since $d_\pm(A_1) = \pm\frac14$, $d(U) = 0$ and by Elkies' theorem $U$ is diagonal
(recall that $U$ is the unimodular overlattice of $L \oplus A_1$). By Lemma~\ref{l:diagonal} $L$ is diagonal.

Suppose now $d_\pm(L) \neq \pm \frac14$, so that in particular $|d_+(L)| = |d_-(L)| =: d \ge \frac74$. Consider the characteristic covector $\xi_0 = e_1 + \dots + e_{2n} \in \Char(D)$, and, for each $i$, the characteristic covector $\xi_i = \xi_0 - 2e_i \in \Char(D)$. Note that $\xi_i$ is norm-minimizing among all characteristic covectors in $D$ for $i = 0,\dots, 2n$, and so that its restrictions $\lambda_i = \rho_\ell(\xi_i) \in L'$ and $\alpha_i = \rho_a(\xi_i)$ are characteristic and they minimize the norm \emph{in their congruence class}. That is, if $\lambda_i \in \Char_{+}(L)$, then $\lambda_i$ minimizes the norm among all elements in $\Char_{+}(L)$; in this case, $\alpha_i \in \Char_-(A)$ and $\alpha_i$ minimizes the norm in $\Char_-(A)$.

Without loss of generality, let us suppose that $\lambda_0 \in \Char_+(L)$. The key observation is that $\lambda_i \in \Char_-(L)$ for each $i = 1,\dots,2n$.
This follows from the fact, observed above, that $e_i \not\in L$, so that $\pi_\ell(2e_i) \in 2L'\setminus 2L$, and in particular $\pi_\ell(2e_i)$ swaps $\Char_+(L)$ and $\Char_-(L)$.

Now, since $\lambda_i \in \Char_-(L)$ for each $i >0$ is a norm-minimizer in its class:
\[
|\lambda_0^2 - \lambda_i^2| = 8d \ge 14.
\]
However,
\[
\lambda_0^2 - \lambda_i^2 = 2\langle \lambda_0, r_i\rangle - r_i^2,
\]
so that for each $i > 0$:
\[
|\langle \lambda_0, r_i)\rangle| \ge 4d-1 \ge 3.
\]

Pick a subset $J \subset \{1, \dots, 2n\}$ of indices such that $R = \{r_j \mid j \in J\}$ is a rational basis for $L$; this in particular means that $|J| = n$. Up to relabelling, let us assume $J = \{1,\dots,n\}$. We claim that $G(R)$ is bicolorable.

To see this, we will prove that all cycles in $G(R)$ have even length, and in fact they all have length $4$. Assume by contradiction that there is a cycle $C \subset G(R)$ of length $k \ge 3, k \neq 4$. Up to another relabelling, let us assume that $C$ comprises $r_1,\dots,r_k$ in this order. Up to replacing $r_i$ with $-r_i$ for some values of $i$, we can assume that all edges $(r_1,r_2), \dots, (r_{k-1},r_k)$ are labelled with $-1$. Under this assumption, $(r_k, r_1)$ has to be labelled by $+1$, for otherwise $(r_1+\dots+r_k)^2 = 0$, which would contradict the fact that $R$ is a linearly independent set. Now recall that $R \subset L \subset D \cong I^{2n}$ comprises elements of square $2$. So there is a basis $f_1,\dots,f_{2n}$ of $D$ such that $r_1 = f_1 - f_2, \dots, r_{k-1} = f_{k-1} - f_k$, and $r_k = f_k + f_1$; but then $r_1 + \dots + r_k = 2f_1 \in L$, which implies $f_1 \in L$ since $L\otimes \Q \cap D = L$, and this contradicts the minimality of $L$.

Since $G(R)$ is bicolorable there is a subset $R' \subset R$, indexed by $J' \subset J$, containing $\lceil \frac n2 \rceil$ roots that are pairwise orthogonal.

Now, by Bessel's inequality:
\[
\lambda_0^2 \ge \sum_{j \in J'} \frac{\langle \lambda_0, r_j \rangle^2}{r_j^2} \ge \left\lceil \frac n2 \right\rceil \cdot \frac92 > 2n,
\]
which contradicts the fact that $\lambda_0^2 \le \lambda_0^2 + \alpha_0^2 = \xi_0^2 = 2n$.
\end{proof}

\begin{remark}
Note that the assumption that the length of the cycle is not $4$ is in fact used: for the $4$-cycle as above, the embedding given by $(f_1-f_2, f_2-f_3, -f_1-f_2, -f_1+f_4)$ has components that sum to $-f_1-f_2-f_3+f_4$.
\end{remark}

\end{section}

\begin{section}{The obstruction}\label{obstruction}

In this section we discuss a topological application of Theorem~\ref{p:2-elkies}.
We start with an algebraic topology lemma.

\begin{lemma}\label{extension}
Let $X$ be a compact, oriented $4$-manifold with boundary $Y$, a closed $3$-manifold with $H_1(Y)$ finite of square-free order.
Then $|{\det Q_X}| = |H_1(Y)|$ and all spin$^c$ structures on $Y$ extend to $X$.
\end{lemma}


\begin{proof}
Let us look at the long exact sequence for the pair $(X, Y)$:
\[
0 \to H^2(X,Y) \to H^2(X) \to H^2(Y) \to H^3(X,Y) \to H^3(X) \to 0.
\]
All spin$^c$ structures on $Y$ extend if and only if the restriction map $H^2(X) \to H^2(Y)$ is onto.
Since $H^2(Y)$ is finite, $H^3(X,Y)$ and $H^3(X)$ have the same rank, $b_3$;
call $B$ and $A$ their torsion subgroups, respectively.
For the same reason, $H^2(X,Y)$ and $H^2(X)$ have the same rank, $b_2$; 
by the universal coefficient theorem and Poincar\'e--Lefschetz duality, their torsion subgroups are isomorphic to $A$ and $B$, respectively.
Since torsion can only map to torsion, call $\tau_i$ the map obtained by restricting $\pi_i^*: H^i(X,Y) \to H^i(X)$ to the torsion subgroup, and then projecting the target to the torsion subgroup;
we regard $\tau_2$ as a map $\tau_2: A\to B$, and $\tau_3$ as a map $\tau_3: B\to A$.
Note that, since $\pi_3^*$ is onto and $H^2(Y)$ is torsion, the induced map on the quotient $H^3(X,Y)/B \to H^3(X)/A$ is injective (in fact, an isomorphism).
Finally, the map $H^2(X,Y)/A\to H^2(X)/B$ is represented by the intersection form $Q_X$ of $X$.

With this in place, we can then apply the nine lemma to
\[
\xymatrix{
0\ar[r] 	& B\ar[r] \ar[d] 		& H^3(X,Y)\ar[r] \ar[d] 	& H^3(X,Y)/B \ar[d] \ar[r]	& 0\\
0\ar[r] 	& A\ar[r] 			& H^3(X)\ar[r] 			& H^3(X)/A\ar[r]		& 0\\
}
\]
and
\[
\xymatrix{
0\ar[r] 	& A\ar[r] \ar[d]	& H^2(X,Y) \ar[r] \ar[d]	& H^2(X,Y)/A \ar[d]\ar[r] 	& 0\\
0\ar[r] 	& B\ar[r]		& H^2(X) \ar[r]			& H^2(X)/B \ar[r] 		& 0\\
}
\]
to obtain that $\kerr \pi_3^*$ is a group $G$ of order $\left|\kerr\tau_3\right| = |B|/|A|$, and that $\cokerr \pi_2^*$ is a group $H$ of order $\left|\cokerr \tau_2\right|\cdot \left|\cokerr Q_X\right| = (|B|/|A|)\cdot |{\det Q_X}|$.

It follows that we can extract a short exact sequence of finite groups:
\[
0 \to H \to H_1(Y) \to G \to 0,
\]
from which
\[
|H_1(Y)| = |H|\cdot|G| = \frac{|B|^2}{|A|^2}|{\det Q_X}|.
\]
Since we assumed that $H_1(Y)$ has square-free order, we conclude that $|A| = |B|$ and $|H_1(Y)| = |{\det Q_X}|$;
each of these conclusions imply that $H^2(X)\to H^2(Y)$ is onto.
\end{proof}

To a closed, oriented, spin$^c$ rational homology 3-sphere $(Y,\t)$, Ozsv\'ath and Szab\'o~\cite{OSz-annals} associate a family of invariants, collectively called \emph{Heegaard Floer homology};
we will work with two of these, namely $\HFp(Y,\t)$ and $\HFoo(Y,\t)$, and we will recall a few results from~\cite{OSz}.
For convenience, we will work over the field of two elements.


Recall that from Heegaard Floer homology we can extract a rational number $d(Y,\t)$, called the \emph{correction term} of $(Y,\t)$, that is an invariant under spin$^c$ rational homology cobordism~\cite{OSz}, and reduces modulo 2 to the rho-invariant $\rho(Y,\t)$.
The correction term $d(Y,\t)$ is defined to be the minimal grading of any non-zero element in the image of the map $\HFoo(Y,\t) \xrightarrow{\pi} \HFp(Y,\t)$.

In the remainder of the section $Y$ will always denote a closed, oriented 3-manifold with $H_1(Y)\cong \Z/2\Z$; we say that $Y$ is a \emph{homology} $\RP^3$.
Hence $Y$ has exactly two spin$^c$ structures.

Our argument will depend on the value of these correction terms.
First we pin down their value modulo 2.

\begin{proposition}\label{labelling}
We can label the two spin$^c$ structures on $Y$ as $\t_+$ and $\t_-$, so that $d(Y,\t_\pm) \equiv \pm \frac{1}{4} \pmod 2$.
\end{proposition}

We start with a preliminary lemma.

\begin{lemma}\label{extension}
Let $W$ be a cobordism from an integral homology sphere $Z$ to a homology $\RP^3$, $Y$.
Then both spin$^c$ structures on $Y$ extend to $W$.
\end{lemma}

\begin{proof}
Carve an open neighborhood of a path from $Z$ to $Y$ into $W$, to obtain a 4-manifold $X$ with boundary $Y\#(-Z)$.
The statement is equivalent to the fact that both spin$^c$ structures on $\partial X = Y\#(-Z)$ extend to $X$, which follows from Lemma~\ref{extension}.
\end{proof}

Fix a spin$^c$ structure $\t$ on $Y$ and a simply-connected 4-manifold $X$ with spin$^c$ structure $\s$, such that $\partial X= Y$ and $\s |_Y = \t$.
Recall now that the d-invariant $d(Y,\t)$ reduces modulo 2 to the rho-invariant $\rho(Y,\t) \in \Q/2\Z$. The rho-invariant is defined as $\rho(Y,\t) \equiv \frac{c_1(\s)^2-\sigma(X)}{4} \pmod 2$
; it follows from the definition that if $(W,\s)$ is a spin$^c$ cobordism from $(Y,\t)$ to $(Y', \t')$, then $\rho(Y',\t') - \rho(Y,\t) \equiv \frac{c_1(\s)^2-\sigma(W)}{4} \pmod 2$.
In particular, integral homology spheres $Z$ have $\rho(Z,\t) = 0$, since they bound spin manifolds (which have signature divisible by 8, by the van der Blij Lemma~\cite[Section~II.5]{HusemollerMilnor} and which have a spin$^c$ structure with trivial first Chern class).

\begin{proof}[Proof of Proposition~\ref{labelling}]
Since $\rho(Y,\t)$ lifts to $d(Y,\t)$, the statement clearly reduces to showing that $\rho(Y,\t_\pm) \equiv \pm\frac14 \pmod 2$.
This is what we set out to prove, by finding a suitable cobordism from $Y$ to an integral homology sphere.

Pick a knot $K$ in $Y$ such that $[K] \neq 0 \in H_1(Y)$.
Then there exists a slope $\gamma$ such that the result of Dehn surgery along $K$ with slope $\gamma$ is an integral homology sphere $Z_0 := Y_\gamma(K)$.
Let $K_0 \subset Z_0$ denote the dual knot.
Since $H_1(Y)\cong \Z/2\Z$ and $Z_0$ is an integer homology sphere, the surgery on $K_0$ that returns $Y$ must have slope $2/q$ for some odd integer $q$.

We can write $2/q$ as a (negative) continued fraction $2/q = [0,-\frac{q+1}2,-2]$, so that $Y$ can be represented by the surgery diagram in Figure~\ref{f:Z0toZ}.
It is easy to see that the 3-manifold obtained by doing surgery on the $0$- and $-\frac{q+1}2$-framed components is again an integral homology sphere, which we denote by $Z$, and that the cobordism $W$ from $Z$ to $Y$ given by the $-2$-framed 2-handle is negative definite.
Moreover, since $W$ is obtained by attaching a single 2-handle to an integral homology sphere, $H_1(W) = H_3(W) = 0$;
in particular, both spin$^c$ structures on $Y$ extend to $W$.

Since $W$ is the trace of a 2-handle attachment along a knot in an integral homology sphere with framing $-2$, the two spin$^c$ structures $\s_+$ and $\s_-$, with Chern classes $0$ and $2\gamma \in H^2(W;\Z) \equiv \Z\cdot \gamma$ respectively, have $c_1(\s_+)^2 = 0$, $c_1(\s_-^2) = -2$.
By the cobordism formula mentioned above, letting $\t_\pm$ be the restriction of $\s_\pm$ to $Y$, we get:
\begin{align*}
\rho(Y,\t_+) = \frac{c_1(\s_+)^2 - \sigma(W)}4 + \rho(Z,\t) \equiv \frac14 \pmod 2,\\
\rho(Y,\t_-) = \frac{c_1(\s_-)^2 - \sigma(W)}4 + \rho(Z,\t) \equiv -\frac14 \pmod 2,
\end{align*}
thus concluding the proof.
%
%
\end{proof}

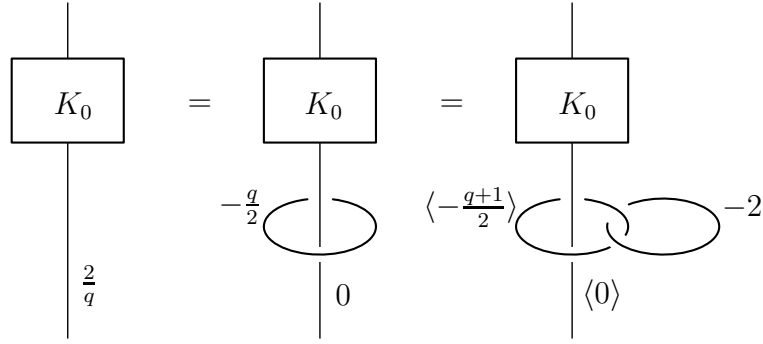
\begin{figure}
\labellist
\pinlabel $K_0$ at 36 133
\pinlabel $=$ at 108 133
\pinlabel $=$ at 252 133
\pinlabel $\frac2q$ at 45 30
\pinlabel $K_0$ at 179 133
\pinlabel $0$ at 190 25
\pinlabel $-\frac{q}2$ at 131 75
\pinlabel $K_0$ at 324 133
\pinlabel $\langle 0\rangle$ at 338 25
\pinlabel $\langle -\frac{q+1}2\rangle$ at 263 75
\pinlabel $-2$ at 418 75
\endlabellist
\centering
\includegraphics[scale=.66]{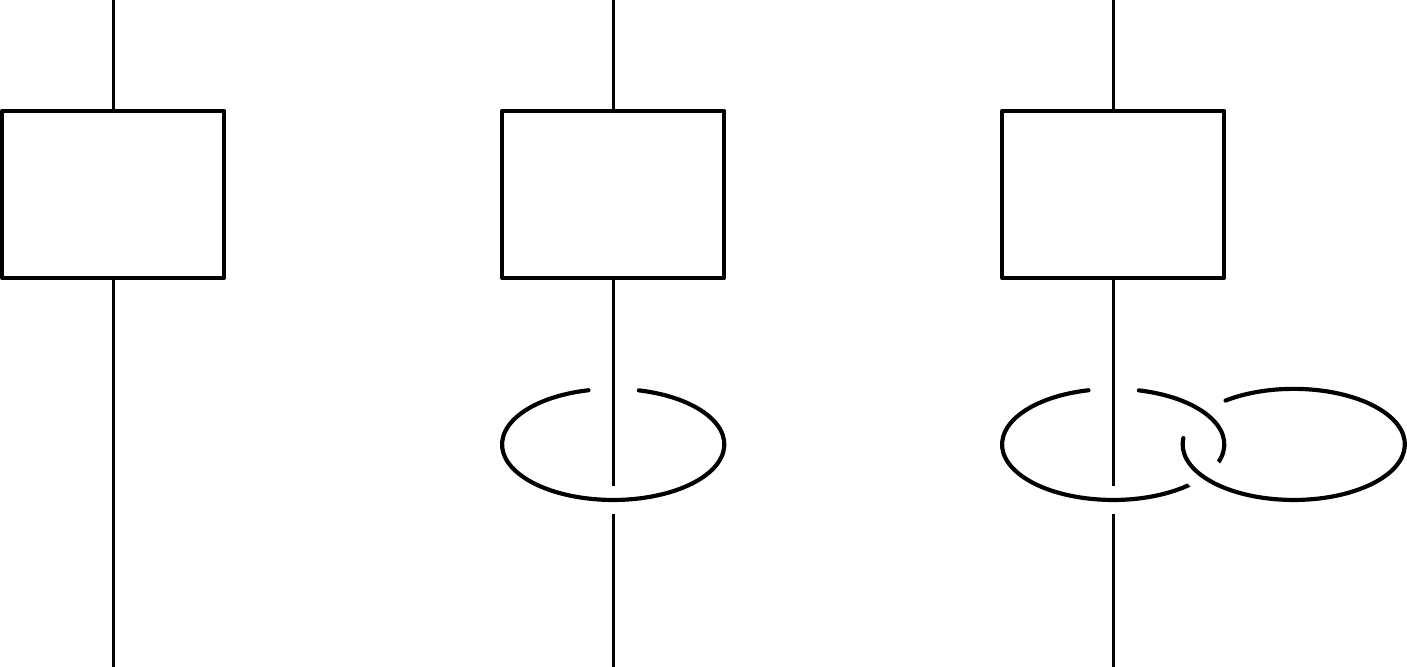}
\caption{Going from $Z_0$ to $Z$.
Recall that the knot $K_0$ lives in $Z_0$.
The right-most picture represents the cobordism from $Z$ to $Y$.
Here we used the braced framing notation: namely, the surgery diagram comprising the components with braced framings describes $Z$, i.e. the lower boundary component of the cobordism, and the non-braced ones represent actual handle attachments for the cobordism.}\label{f:Z0toZ}
\end{figure}
%
%

Proposition~\ref{labelling} justifies the following definition.

\begin{definition}
For $Y$ a homology $\RP^3$, we set $d_{\pm 1/4}(Y) = d(Y,\t_\pm)$.
\end{definition}

Note that the labelling is chosen so that $d_{\pm 1/4}(Y) \equiv \pm \frac14 \pmod 2$;
observe also that since $d(Y,\t) = -d(-Y,\t)$, we have that $d_{\pm1/4}(-Y) = -d_{\mp 1/4}(Y)$.

Now suppose that $Y$ bounds a \emph{positive definite} 4-manifold $W$.
In this context we have the following inequality.

\begin{theorem}[\cite{OSz}]\label{t:negativedefiniteinequality}
For each spin$^c$ structure $\s$ on $W$ with $\s|_Y = \t$, we have
\[
\frac{c_1(\s)^2 -b_2(W)}4 \geq d(Y,\t).
\]
Moreover, the two sides of the inequality are congruent modulo $2$.
\end{theorem}

%



We are now ready to give a topological translation of Theorem~\ref{p:2-elkies}.
\begin{proposition}\label{ob}
Let $Y$ be a homology $\RP^3$.
If $Y$ bounds a positive definite $4$-manifold, then $d_{1/4}(Y) + d_{-1/4}(Y) \le 0$.
Moreover, if equality is attained, then $d_{\pm 1/4}(Y) = \pm \frac{1}{4}$.
\end{proposition}

\begin{proof}
Suppose that $Y$ bounds a positive definite $4$-manifold $W$, and let $L$ be the lattice $(H_2(W;\Z)/{\rm Tor}, Q_W)$.
By Lemma~\ref{extension}, $L$ is a positive definite lattice of determinant $2$, and the first Chern class gives a surjection $c_1 \colon \Spinc(W) \to \Char(L)$. Call $n = b_2(W) = \rk L$.

By the last statement in Theorem~\ref{t:negativedefiniteinequality}, using the labelling of Proposition~\ref{labelling}, we see that $\s \in \Spinc(W)$ restricts to $\t_\pm$ if and only if $c_1(\s) \in \Char_\pm(L)$.
Let $\xi_+ \in \Char_+(L)$ and $\xi_- \in \Char_-(L)$ be characteristic covectors with minimal square; note that there exist spin$^c$ structures $\s_\pm$ on $W$ such that $c_1(\s_\pm) = \xi_\pm$, and that $\s_\pm$ restricts to $\t_\pm$.

Then using Theorem~\ref{t:negativedefiniteinequality} and Theorem~\ref{p:2-elkies} we get
\[
d_{1/4}(Y) + d_{-1/4}(Y) \le \frac{\xi_+^2 -n}4 + \frac{\xi_-^2 -n}4 = d_+(L) + d_-(L) \le 0,
\]
proving the first part of the theorem.
Furthermore, if $d_{1/4}(Y) + d_{-1/4}(Y) = 0$, then the above inequality forces $d_+(L) + d_-(L) = 0$, and so by Theorem~\ref{p:2-elkies} we get that $d_\pm(L) = \pm \frac14$.
This, in turn, together with Theorem~\ref{t:negativedefiniteinequality}, forces $d_{\pm1/4}(Y) =  \pm \frac14$.
%
%
%
\end{proof}

\end{section}
\begin{section}{The example}\label{example}

Recall that we defined $\overline Y$ as the Seifert fibered space $Y(2;\frac{15}{13},\frac{17}{3},\frac{23}{22})$ and $N = 3P \# \overline Y$, where $P$ is the Poincar\'e homology sphere, oriented as the boundary of the negative $E_8$-plumbing;
equivalently, $P$ is the Brieskorn sphere $\Sigma(2,3,5)$.

We start by computing the correction terms of $\overline Y$.

\begin{proposition}
$d_{-1/4}(\overline Y) = -\frac{17}{4}$ and $d_{1/4}(\overline{Y}) = -\frac{31}{4}$.
\end{proposition}

\begin{proof}
Since $-\overline{Y}$ is the boundary of a negative definite plumbing with a single bad vertex, we can compute these correction terms using \c{C}a\u{g}r{\i} Karakurt's implementation~\cite{Cagricode} of N\'emethi's formula~\cite[Section~11.13]{Nemethi}, which, in turn, is a generalization of Ozsv\'ath and Szab\'o's algorithm from~\cite{OSz3}.
\end{proof}

\begin{remark}
N\'emethi computes of the $d$-invariant of a Seifert fibered space as a sum of two terms;
the first summand is expressed in terms of Dedekind--Rademacher sums associated to the Seifert parameters~\cite[Section 11.9]{Nemethi}, while the second depends on the minimum of a certain function $\tau: \Z_{\ge 0} \to \Z$.
The function is eventually increasing, and the minimum is contained in a bounded interval $[0,N]$, where $N$ can be chosen to be the product of the multiplicities of the fibers.

Furthermore, in principle the computation of the correction terms of $\overline Y$ could be done in other ways: either by computing the minimal squares in the lattice associated to the canonical negative plumbing of $-\overline Y$~\cite[Corollary~1.5]{OSz3} or by following the entire algorithm in~\cite{OSz3}.
\end{remark}

We are now ready to prove our main result; more precisely, we will prove that $N$ does not bound a definite 4-manifold.

\begin{proof}[Proof of Theorem~\ref{main}]
By additivity of correction terms, and since $d(P,\t) = 2$ for the unique spin$^c$ structure $\t$ on $P$, we know that $d_{\pm 1/4}(N) = 3d(P,\t) + d_{\pm 1/4}(\overline{Y})$. By the previous proposition, we get $d_{\pm 1/4}(N) = \mp \frac74$.

Proposition~\ref{ob} now implies that $N$ cannot bound a positive definite 4-manifold.
Reversing orientation and again applying Proposition~\ref{ob} shows that $N$ cannot bound a negative definite 4-manifold either.
\end{proof}

We conclude with two observations about $\overline Y$ and spineless 4-manifolds.

\begin{proposition}
Let $\Sigma$ be an integral homology sphere. The $3$-manifold $\overline Y \# \Sigma$ is not integral homology cobordant to a $3$-manifold obtained as Dehn surgery along a knot in $S^3$.
\end{proposition}

\begin{proof}
If we let $Y = S^3_{2/q}(K)$ with $q > 0$; then, by~\cite[Proposition~1.4 and Lemma~2.4]{NW}:
\begin{align*}
d_{\pm 1/4}(Y) \in \left\{-2V_0(K) + \frac14, -2V_1(K) + \frac14\right\} \Longrightarrow d_{1/4}(Y) &\ge -2V_0(K) + \frac14,\\
 d_{-1/4}(Y) &= -2V_0(K) - \frac14,
\end{align*}
so that in particular $d_{1/4}(Y) - d_{-1/4}(Y) \ge \frac12$. However
\[ d_{1/4}(\overline Y \# \Sigma) - d_{-1/4}(\overline Y \# \Sigma) = d_{1/4}(-(\overline Y \# \Sigma)) - d_{-1/4}(-(\overline Y \# \Sigma)) = -\frac72,\]
so $\pm(\overline Y\#\Sigma)$ cannot be integrally homology cobordant to a positive surgery along a knot in $S^3$.
\end{proof}

The following remark was suggested to the authors by Adam Levine.

\begin{remark}
Note that the previous proposition implies that, for any integral homology sphere $\Sigma$, $\overline Y \# \Sigma$ cannot bound a simply-connected spineless 4-manifold, i.e. a 4-manifold $W$ that is homotopy equivalent to $S^2$ but such that the generator of $H_2(W)$ is not represented by a PL-sphere.

Closely following Levine and Lidman's approach~\cite{LevineLidman}, we produce a homotopy $S^2$ whose boundary is $\overline Y \# \Sigma$ for some homology sphere $\Sigma$, which is necessarily going to be spineless. We sketch the construction, which is very similar to~\cite{LevineLidman}.

The key observation is that there is an integral homology sphere $-\Sigma$ such that $\overline Y$ is obtained as integral surgery along a knot in $-\Sigma$. For example, we can choose $\Sigma$ to be the Brieskorn sphere $\Sigma(15,17,181)$. Indeed, the negative plumbing graph of $\Sigma(15,17,181)$ is obtained by adding a single vertex to the negative plumbing graph of $-\overline Y$, which exhibits $\overline Y$ as surgery along a singular fiber of $-\Sigma(15,17,181)$.

By~\cite[Lemma~3.2 and Proposition~3.1]{LevineLidman}, the 4-manifold obtained from the trace of this surgery and carving a path in $\Sigma \times I$ is a homotopy $S^2$ whose boundary is $\overline Y \# \Sigma$.
\end{remark}

\end{section}

\bibliographystyle{abbr}
\bibliography{lattices}

\begin{bibdiv}
\begin{biblist}

\bib{CochranTanner}{article}{
      author={Cochran, Tim~D.},
      author={Tanner, Daniel},
       title={Homology cobordism and {S}eifert fibered 3-manifolds},
        date={2014},
        ISSN={0002-9939},
     journal={Proc. Amer. Math. Soc.},
      volume={142},
      number={11},
       pages={4015\ndash 4024},
         url={https://doi.org/10.1090/S0002-9939-2014-12122-8},
      review={\MR{3251741}},
}

\bib{Elkies}{article}{
      author={Elkies, Noam~D.},
       title={A characterization of the {${\bf Z}^n$} lattice},
        date={1995},
        ISSN={1073-2780},
     journal={Math. Res. Lett.},
      volume={2},
      number={3},
       pages={321\ndash 326},
         url={https://doi.org/10.4310/MRL.1995.v2.n3.a9},
      review={\MR{1338791}},
}

\bib{Froyshov2}{unpublished}{
      author={Fr{\o}yshov, Kim~A.},
       title={Mod 2 instanton homology},
        note={In preparation},
}

\bib{Froyshov1}{article}{
      author={Fr{\o}yshov, Kim~A.},
       title={Equivariant aspects of {Y}ang-{M}ills {F}loer theory},
        date={2002},
        ISSN={0040-9383},
     journal={Topology},
      volume={41},
      number={3},
       pages={525\ndash 552},
         url={https://doi.org/10.1016/S0040-9383(01)00018-0},
      review={\MR{1910040}},
}

\bib{HaydenPiccirillo}{unpublished}{
      author={Hayden, Kyle},
      author={Piccirillo, Lisa},
       title={The trace embedding lemma and spinelessness},
        date={2019},
        note={To appear in J. Differential Geom. Preprint available at
  \href{http://arxiv.org/abs/1912.13021}{arXiv:1912.13021}},
}

\bib{HKMP}{article}{
      author={Hedden, Matthew},
      author={Kim, Min~Hoon},
      author={Mark, Thomas~E.},
      author={Park, Kyungbae},
       title={Irreducible 3-manifolds that cannot be obtained by 0-surgery on a
  knot},
        date={2019},
        ISSN={0002-9947},
     journal={Trans. Amer. Math. Soc.},
      volume={372},
      number={11},
       pages={7619\ndash 7638},
         url={https://doi.org/10.1090/tran/7786},
      review={\MR{4029676}},
}

\bib{HKL}{article}{
      author={Hom, Jennifer},
      author={Karakurt, \c{C}a\u{g}r{\i}},
      author={Lidman, Tye},
       title={Surgery obstructions and {H}eegaard {F}loer homology},
        date={2016},
        ISSN={1465-3060},
     journal={Geom. Topol.},
      volume={20},
      number={4},
       pages={2219\ndash 2251},
         url={https://doi.org/10.2140/gt.2016.20.2219},
      review={\MR{3548466}},
}

\bib{Cagricode}{unpublished}{
      author={Karakurt, \c{C}a\u{g}r{\i}},
       title={Magma code},
  note={\href{http://web0.boun.edu.tr/cagri.karakurt/HFNem2.zip}{http://web0.boun.edu.tr/cagri.karakurt/HFNem2.zip}},
}

\bib{Larson}{incollection}{
      author={Larson, Kyle},
       title={Lattices and correction terms},
        date={2018},
      series={Trends in Mathematics},
   publisher={Birkh{\"a}user},
        note={To appear in {S}ingularities and {T}heir {I}nteraction with
  {G}eometry and {L}ow {D}imensional {T}opology. Preprint available at
  \href{https://arxiv.org/abs/1807.05098}{arXiv:1807.05098}},
}

\bib{LevineLidman}{article}{
      author={Levine, Adam~Simon},
      author={Lidman, Tye},
       title={Simply connected, spineless 4-manifolds},
        date={2019},
     journal={Forum Math. Sigma},
      volume={7},
       pages={Paper No. e14, 11},
         url={https://doi.org/10.1017/fms.2019.11},
      review={\MR{3952511}},
}

\bib{Livingston}{article}{
      author={Livingston, Charles},
       title={Homology cobordisms of {$3$}-manifolds, knot concordances, and
  prime knots},
        date={1981},
        ISSN={0030-8730},
     journal={Pacific J. Math.},
      volume={94},
      number={1},
       pages={193\ndash 206},
         url={http://projecteuclid.org/euclid.pjm/1102735917},
      review={\MR{625818}},
}

\bib{HusemollerMilnor}{book}{
      author={Milnor, John},
      author={Husemoller, Dale},
       title={Symmetric bilinear forms},
   publisher={Springer-Verlag, New York-Heidelberg},
        date={1973},
        note={Ergebnisse der Mathematik und ihrer Grenzgebiete, Band 73},
      review={\MR{0506372}},
}

\bib{Myers}{article}{
      author={Myers, Robert},
       title={Homology cobordisms, link concordances, and hyperbolic
  {$3$}-manifolds},
        date={1983},
        ISSN={0002-9947},
     journal={Trans. Amer. Math. Soc.},
      volume={278},
      number={1},
       pages={271\ndash 288},
         url={https://doi.org/10.2307/1999315},
      review={\MR{697074}},
}

\bib{Nemethi}{article}{
      author={N\'emethi, Andr\'as},
       title={On the {O}zsv{\'a}th--{S}zab{\'o} invariant of negative definite
  plumbed 3-manifolds},
        date={2005},
     journal={Geom. Topol.},
      volume={9},
      number={2},
       pages={991\ndash 1042},
}

\bib{NW}{article}{
      author={Ni, Yi},
      author={Wu, Zhongtao},
       title={Cosmetic surgeries on knots in {$S^3$}},
        date={2015},
        ISSN={0075-4102},
     journal={J. Reine Angew. Math.},
      volume={706},
       pages={1\ndash 17},
         url={https://doi.org/10.1515/crelle-2013-0067},
      review={\MR{3393360}},
}

\bib{NST}{unpublished}{
      author={Nozaki, Yuta},
      author={Sato, Kouki},
      author={Taniguchi, Masaki},
       title={Filtered instanton floer homology and the homology cobordism
  group},
        date={2019},
        note={Preprint available at
  \href{http://arxiv.org/abs/1905.04001}{arXiv:1905.04001}},
}

\bib{OwensStrle-char}{article}{
      author={Owens, Brendan},
      author={Strle, Sa\v{s}o},
       title={A characterization of the {$\Bbb Z^n\oplus\Bbb Z(\delta)$}
  lattice and definite nonunimodular intersection forms},
        date={2012},
        ISSN={0002-9327},
     journal={Amer. J. Math.},
      volume={134},
      number={4},
       pages={891\ndash 913},
         url={https://doi.org/10.1353/ajm.2012.0026},
      review={\MR{2956253}},
}

\bib{OSz}{article}{
      author={Ozsv\'ath, Peter},
      author={{{\relax Sz}ab\'o}, Zolt\'an},
       title={Absolutely graded {F}loer homologies and intersection forms for
  four-manifolds with boundary},
        date={2003},
        ISSN={0001-8708},
     journal={Adv. Math.},
      volume={173},
      number={2},
       pages={179\ndash 261},
         url={http://dx.doi.org/10.1016/S0001-8708(02)00030-0},
      review={\MR{1957829}},
}

\bib{OSz3}{article}{
      author={Ozsv\'ath, Peter},
      author={{{\relax Sz}ab\'o}, Zolt\'an},
       title={On the {F}loer homology of plumbed three-manifolds},
        date={2003},
        ISSN={1465-3060},
     journal={Geom. Topol.},
      volume={7},
       pages={185\ndash 224 (electronic)},
         url={http://dx.doi.org/10.2140/gt.2003.7.185},
      review={\MR{1988284 (2004h:57039)}},
}

\bib{OSz-annals}{article}{
      author={Ozsv\'ath, Peter},
      author={{{\relax Sz}ab\'o}, Zolt\'an},
       title={Holomorphic disks and topological invariants for closed
  three-manifolds},
        date={2004},
        ISSN={0003-486X},
     journal={Ann. of Math. (2)},
      volume={159},
      number={3},
       pages={1027\ndash 1158},
         url={https://doi.org/10.4007/annals.2004.159.1027},
      review={\MR{2113019}},
}

\bib{Stoffregen}{article}{
      author={Stoffregen, Matthew},
       title={Pin(2)-equivariant {S}eiberg-{W}itten {F}loer homology of
  {S}eifert fibrations},
        date={2020},
        ISSN={0010-437X},
     journal={Compos. Math.},
      volume={156},
      number={2},
       pages={199\ndash 250},
         url={https://doi.org/10.1112/s0010437x19007620},
      review={\MR{4044465}},
}

\end{biblist}
\end{bibdiv}

\end{document}